\documentclass[11pt]{article}
\usepackage{amsmath}
\usepackage{amsthm}
\input{epsf.tex}

\newtheorem{theorem}{Theorem}[section]
\newtheorem{proposition}[theorem]{Proposition}
\newtheorem{lemma}[theorem]{Lemma}

\def\Empty{}

\catcode`\@=11
\def\section{\@startsection {section}{1}{\z@}{-3.5ex plus -1ex minus 
-.2ex}{2.3ex plus .2ex}{\large\bf}}

\def\fnum@figure{{\small Figure \thefigure}}
\def\fakefigure{\def\@captype{figure}}

\long\def\@makecaption#1#2{
    \vskip 10pt 
    \def\FCap{#2} \def\NoCap{\ignorespaces}
    \ifx \FCap\NoCap
       \setbox\@tempboxa\hbox{#1}  % This is to avoid the damn colon.
      \else
       \setbox\@tempboxa\hbox{#1: \small \it #2}
    \fi
    \ifdim \wd\@tempboxa >\hsize   % IF longer than one line:
        \unhbox\@tempboxa\par      %   THEN set as ordinary paragraph.
      \else                        %   ELSE  center.
        \hbox to\hsize{\hfil\box\@tempboxa\hfil}  
    \fi}

\def\@oddhead{\hbox{}\rightmark \hfil \rm\thepage}% Right heading.
\def\sectionmark#1{\markright {\sc{\ifnum \c@secnumdepth >\z@
      \S\thesection.\hskip 1em\relax \fi #1}}}

\catcode`\@=12

\def\oplabel#1{
  \def\OpArg{#1} \ifx \OpArg\Empty {} \else
  	\label{#1}
  \fi}

\newlength{\saveu}

\catcode`\@=11

\newcommand{\uu}{\mbox{${\cal U}$}}

\newcommand{\hhs}{\mbox{${\cal H} ^s$}}

\newcommand{\hp}{\mbox{${\cal H}$}}

\newcommand{\pin}{\mbox{$\partial _{\infty}$}}

\newcommand{\oo}{\mbox{$\cal O$}}
\newcommand{\tz}{\mbox{$\cal T$}}
\newcommand{\tzs}{\mbox{${\cal T}^s$}}
\newcommand{\tzu}{\mbox{${\cal T}^u$}}
\newcommand{\oos}{\mbox{${\cal O}^s$}}
\newcommand{\oou}{\mbox{${\cal O}^u$}}

\newcommand{\rrrr}{\mbox{${\bf R}$}}

\newcommand{\mi}{\mbox{$\widetilde M$}}

\newcommand{\wws}{\mbox{$\widetilde W^s$}}

\newcommand{\wwp}{\mbox{$\widetilde \Phi$}}

\newcommand{\fol}{\mbox{$\cal F$}}
\newcommand{\gal}{\mbox{$\cal G$}}

\newcommand{\gn}{\mbox{$\widetilde {\cal G}$}}

\newcommand{\fs}{\mbox{${\cal F} ^s$}}
\newcommand{\fu}{\mbox{${\cal F} ^u$}}

\newcommand{\fnsE}{\mbox{$\widetilde {\cal F}^s_E$}}
\newcommand{\fnuE}{\mbox{$\widetilde {\cal F}^u_E$}}
\newcommand{\lsl}{\mbox{$\widetilde {\cal L}^s_L$}}
\newcommand{\lul}{\mbox{$\widetilde {\cal L}^u_L$}}
\newcommand{\fnsL}{\mbox{$\widetilde {\cal F}^s_L$}}

\newcommand{\fnsLi}{\mbox{$\widetilde {\cal F}^s_{L_i}$}}
\newcommand{\fnuL}{\mbox{$\widetilde {\cal F}^u_L$}}

\newcommand{\fns}{\mbox{${\widetilde {\cal F}}^s $}}

\newcommand{\fnu}{\mbox{${\widetilde {\cal F}}^u $}}

\newcommand{\ws}{\mbox{$\widetilde  W^s$}}
\newcommand{\wu}{\mbox{$\widetilde  W^u$}}

\def\@evenhead{\rm \leftmark \hfil \thepage}%        Left heading.
\def\chaptermark#1{\markboth {\sc {\ifnum \c@secnumdepth >\m@ne
      \@chapapp\ \thechapter. \ \fi #1}}{}}%

\catcode`\@=12

\begin{document}

\title{Rigidity of pseudo-Anosov flows transverse to $\rrrr$-covered foliations}
\author{S\'{e}rgio R. Fenley
\thanks{Reseach partially supported by NSF grant 
DMS-0305313.}
}
\maketitle

\vskip .2in

{\small{
\noindent
{\bf {Abstract}} $-$ 
A foliation is $\rrrr$-covered if the leaf space in the universal cover
is homeomorphic to the real numbers.
We show that, up to topological conjugacy,
there are at most two pseudo-Anosov flows transverse such a foliation.
If there are two, then the foliation is weakly conjugate to the
stable foliation of an $\rrrr$-covered Anosov flow.
The proof uses the universal circle for $\rrrr$-covered foliations.
}}

\vskip .2in
\section{Introduction}
Pseudo-Anosov flows  are extremely common amongst $3$-manifolds 
\cite{GK1,Mo2,Fe2,Cal2,Cal3} and they yield important topological
and geometrical information about the manifold.
For example they imply that the manifold is irreducible and the universal
cover is homeorphic to $\rrrr^3$ \cite{Ga-Oe}. There are also
relations
with the atoroidal property \cite{Fe3}. Finally there are consequences
for the large scale geometry of the universal cover
when the manifold is atoroidal:
In that case it follows that the fundamental group is Gromov 
hyperbolic \cite{GK2} and in certain cases the dynamics structure of the flow
produces a flow ideal
boundary to the universal cover which is equivalent to the Gromov boundary
and yields many geometric results \cite{Fe7}.

As for the existence of pseudo-Anosov flows, it turns out that
many classes of Reebless, foliations in atoroidal $3$-manifolds
admit transverse or almost transverse pseudo-Anosov which are
constructed using the foliation structure: \
1) fibrations over the circle \cite{Th1},  \ 2) finite depth foliations \cite{Mo2},
\ 3) $\rrrr$-covered foliations \cite{Fe2,Cal2} \ and 4) Foliations
with one sided branching \cite{Cal3}. Pseudo-Anosov flows also
survive under the majority of Dehn surgeries on closed orbits
\cite{GK1}, which makes them extremely common. 
On the other hand there are a few examples of
non existence of
pseudo-Anosov flows in certain specific manifolds: see \cite{Br} for
examples in Seifert fibered spaces and  \cite{Ca-Du,Fe5} for examples
in hyperbolic manifolds.

In this article we consider the uniqueness question for such flows:
Up to topological conjugacy, how many pseudo-Anosov flows are there in
a closed $3$-manifold?
Topological conjugacy means that there is a homeomorphism of the
manifold sending orbits to orbits.
% and preserving the flow direction.
The less flows there are, the more rigid these flows are and
consequently more likely to give information about the manifold.
In this generality the question is, at this point, very hard to
tackle. 
Here we start the study of this question and 
we consider how many pseudo-Anosov flows are there
transverse to a given foliation.
This is very natural, since as explained above, many pseudo-Anosov flows
are constructed from the structure of a given foliation and are
transverse to it.
We will consider a certain class of foliations called {\em $\rrrr$-covered}:
this means that the leaf space in the universal cover is  homeomorphic to
the set of real numbers \cite{Fe2}.
This is the simplest situation with respect to this question.
The uniqueness analysis involves a detailed understanding of the
topology and geometry of the foliation and flow in this case.

There are many examples of $\rrrr$-covered foliations: 1) Fibrations over
the circle; 2) Many stable and unstable foliations of Anosov flows,
which are then called $\rrrr$-covered Anosov flows. These include 
geodesic flows of hyperbolic surfaces and many examples in
hyperbolic $3$-manifolds \cite{Fe1};
3) {\em Uniform} foliations \cite{Th2}: this means that given any two leaves of the
lifted foliation in the universal cover, they are a bounded distance from
each other. Obviously the bound depends on the pair of leaves. This is 
associated with slitherings over the circle \cite{Th2};
4) Many examples of $\rrrr$-covered but not uniform foliations
in hyperbolic $3$-manifolds \cite{Cal1}.

We should remark that in this article pseudo-Anosov flows include flows without
singularities, that is (topological) Anosov flows.
On the other hand, we do not allow $1$-prong singularities. With $1$-prongs
almost all control is lost, for example ${\bf S}^2 \times {\bf S}^1$ has
a pseudo-Anosov flow with $1$-prongs and the manifold is not even irreducible.

A flow transverse to a foliation is {\em regulating} if an arbitrary orbit in
the universal cover intersects every leaf of the lifted foliation.
In particular this implies that the foliation is $\rrrr$-covered.
This is strongly related to the atoroidal property:
Given an $\rrrr$-covered foliation with a transverse, regulating  pseudo-Anosov
flow, it follows that either the manifold is atoroidal \cite{Fe3}
or it fibers over the circle with fiber a torus and Anosov monodromy.
Conversely if the manifold is atoroidal and acylindrical, then there is a regulating,
pseudo-Anosov
flow transverse to the $\rrrr$-covered foliation \cite{Fe2,Cal2}.
So transverse pseudo-Anosov flows are as general as possible in this situation and
the uniqueness question is a very natural one in this setting.

There is one case where the uniqueness question for transverse flows is known,
which is the simplest case of foliations:
a fibration over the circle. 
It is easy to see that any transverse
flow is regulating.
Any two transverse flows induce homotopic and hence isotopic 
monodromies of the fiber $S$.
This works even if the flow is not pseudo-Anosov. 
If the flow is pseudo-Anosov, then the associated monodromy is a
pseudo-Anosov homeomorphism of $S$ \cite{Th1}.
In particular the fiber cannot be the sphere or the projective plane.
If the fiber is Euclidean, then the flow has no singularities
and is a topological Anosov flow. In this case it is not hard to prove
that there is at most one transverse pseudo-Anosov flow up to conjugacy.
Suppose then that the fiber is hyperbolic and therefore the monodromy
is pseudo-Anosov with singularities.
It is proved in \cite{FLP}, expos\'{e} 12,
 that any two homotopic pseudo-Anosov homeomorphisms are
in fact conjugate. This implies that the corresponding flows are also topologically
conjugate and consequently in this case there is only one transverse pseudo-Anosov flow up
to conjugacy.

This result turns out to be very close to what
happens in general for $\rrrr$-covered foliations:

\vskip .1in
\noindent
{\bf {Theorem}} $-$ Let $\gal$ be an $\rrrr$-covered foliation in $M^3$ closed.
Then up to topological conjugacy there is at most one transverse pseudo-Anosov flow which 
is regulating for $\gal$. In addition, up to conjugacy, there is also at most one non regulating
transverse pseudo-Anosov flow to $\gal$.
If there is a transverse pseudo-Anosov flow  which is non regulating for $\gal$, 
then the flow has no singular orbits and is a topological Anosov flow.
In addition in  this case, after a blow down of
foliated $I$-bundles, then the resulting foliation $\gal'$ is conjugate to either the stable or the unstable
foliation
of the Anosov flow.
\vskip .1in

Consequently if $\gal$ is not a blow up of the stable/unstable foliation of an 
$\rrrr$-covered Anosov flow then up to topological conjugacy,
there is at most one pseudo-Anosov flow transverse to $\gal$.
Furthermore there is one such flow if $M$ is atoroidal.

A {\em foliated $I$-bundle} of $\gal$ is an $I$-bundle $V$ embedded in $M$ so that
$V$ is a union of leaves of $\gal$, which are transverse to the $I$-fibers
in $V$. In particular the boundary of $V$ is an union of leaves of $\gal$.
In general the base of the bundle is not a compact surface.
The blow down operation collapses a foliated $I$-bundle onto a single leaf,
by collapsing $I$-fibers to points.
In the theorem  above
one may need to do this blow down operation a countable number of times.
With reference to the abstract of this article, the phrase $\gal$ is weakly conjugate to 
a foliation $\fol$, means that some blow down $\gal'$ of $\gal$ is 
topologically conjugate to $\fol$.

This theorem generalizes the result for fibrations, because as explained above in that case
any transverse flow is regulating.

\vskip .1in
In order to prove the theorem we split into two cases: the regulating and non
regulating situations.
The non regulating case was studied in \cite{Fe4} where the second part of the
theorem is proved. Here is an outline of that result.
The result uses the topological theory of pseudo-Anosov flows,
see \cite{Fe4,Fe6}.
In the universal cover $\mi$ of $M$, the lifted
flow
has stable and unstable foliations. Since 
$\gal$ is $\rrrr$-covered there is only one transverse
direction to the lift $\gn$ of the foliation $\gal$ to $\mi$.
After a considerable
analysis, this implies that there is only one transverse direction to the
stable and unstable foliations of the flow
in the universal cover. In particular we show that there are no
singularities of the flow $-$ it is a (topological) Anosov flow. In addition we
prove that 
the stable and unstable foliations of the flow $-$ which now are non  singular
foliations $-$ are $\rrrr$-covered foliations.
Therefore the flow is an $\rrrr$-covered Anosov flow. 

The next step is to show that for each
leaf of $\gn$ there is a well defined 
stable (or unstable) leaf in the universal cover associated to it and these 
two leaves (one stable/unstable and the other a leaf of $\gn$) are a bounded
Hausdorff distance from each other.
For simplicity assume they are stable leaves.
After collapsing foliated $I$-bundles
of $\gal$, this correspondence between leaves of the stable foliation
in the universal cover
and leaves of $\gn$ is a bijection. 
Since the leaf of $\gn$ and the corresponding stable leaf are a bounded 
Hausdorff distance
from each other, there is a map between them which sends
a point in one leaf to a point at a bounded distance in the other leaf.
As both foliations are $\rrrr$-covered then this map is a quasi-isometry.
Since leaves of the stable foliation are Gromov hyperbolic \cite{Pl,Su}
and any leaf of $\gn$ is quasi-isometric to a stable leaf,
it follows that the leaves of $\gn$ are
also Gromov hyperbolic.
In particular in the non regulating case, there are no parabolic leaves in $\gal$.
In \cite{Fe4} the analysis was done under the assumption that leaves of 
$\gal$ are Gromov hyperbolic. The argument above shows that this assumption is not 
necessary.
%5as it follows from the quasi-isometry with leaves of the stable foliation.
Using a result of Candel \cite{Can}, we can assume that the leaves of
$\gal$ are hyperbolic leaves.

The next step is to show that for each flow line
in a fixed leaf of the stable foliation in the universal cover there is a unique geodesic in
the corresponding leaf of $\gn$, so that they are a bounded Hausdorff distance
from each other. 
These geodesics in leaves of $\gn$ jointly produce a flow, which projects to a flow
in $M$ which is Anosov and whose flow lines are contained in leaves of $\gal$.
The construction shows that the new Anosov flow is conjugate to the original one
and therefore
$\gal$ is topologically conjugate to
the stable foliation of the original Anosov flow.

We remark that it is very easy to construct non regulating examples for
certain foliations: let $\gal$ be
the stable foliation of an $\rrrr$-covered Anosov flow $\Psi$.
Perturb the flow $\Psi$ slightly along the unstable leaves, to produce
a new Anosov flow $\Phi$ which is transverse to $\gal$ and non regulating
for $\gal$.

\vskip .1in
In this article we consider the regulating  situation. 
The proof is completely different from the non regulating case: 
in that case the proof was internal to $\mi$ $-$ we only used
the topology of the pseudo-Anosov flow and showed that stable/unstable
leaves and leaves of $\gn$ are basically parallel to each other.
Clearly this cannot happen in the regulating situation.
In the regulating case we use the asymptotics of the foliation,
contracting directions between leaves, the universal circle for foliations 
and relations of these with the flow.
We show that the universal circle of the foliation
can be thought of as an ideal boundary for the orbit space
of a regulating pseudo-Anosov flow and this can be used to
completely determine the flow from outside in $-$ from the
universal circle 
ideal boundary to the universal cover of the manifold in an 
equivariant way.

The proof of the theorem goes as follows. 
Let 
$\Phi$ be transverse and regulating for the foliation $\gal$.
Suppose first that there is a parabolic leaf.
Then we actually show that there has to be a compact leaf which is parabolic.
Hence 
the manifold
fibers over the circle  with fiber this leaf and the
flow is topologically conjugate to a suspension Anosov flow.
In this case there is at most one pseudo-Anosov flow transverse to $\gal$,
since there cannot be a non regulating transverse pseudo-Anosov flow.
This is done in section \ref{para}.

In the case that all leaves are Gromov hyperbolic, 
use Candel's theorem \cite{Can} 
and assume the leaves are hyperbolic.
The orbit space of a pseudo-Anosov flow is the space
of orbits in the universal cover. It is always homeomorphic to the plane 
\cite{Fe-Mo} and the fundamental group of the manifold acts naturally 
in this orbit space.
Given two regulating pseudo-Anosov flows transverse to $\gal$ we produce
a homeomorphism between the corresponding orbit spaces, which is group equivariant.
This is the main step here.
Using the foliation $\gn$ which is transverse to each lifted flow, this produces a homeomorphism of 
the universal cover of the manifold, which takes orbits of one flow to orbits of the other flow
and is group equivariant. This produces the conjugacy.

In order to produce the homeomorphism between the orbit spaces,
we use in an essential way the universal circle for foliations
as introduced by Thurston \cite{Th2,Th3,Th4}.
For $\rrrr$-covered foliations, the universal circle is canonically identified to the circle
at infinity of any leaf of $\gn$ \cite{Fe2,Cal2}. 
Notice that the universal circle depends only on the foliation and not on
the transverse pseudo-Anosov flow.
We first consider only one pseudo-Anosov flow transverse to $\gal$.
We show that the orbit space of the flow in $\mi$ can be compactified
with the universal circle of the foliation to produce a closed disk.
This is canonically identified with the standard compactification of
of any hyperbolic leaf of $\gn$.
Here one has to show that the orbit space of the flow in $\mi$ and the universal circle
of the foliation are compatible with the topology of the leaves of 
$\gn$ and also that this topology is independent of the particular
leaf of $\gn$.
To prove this fact, 
one has to distinguish between uniform and non uniform foliations.
Recall that uniform means that any two leaves of $\gn$ are a finite Hausdorff distance
from each other $-$ for example fibrations over the circle.
The uniform case is simple. The non uniform
case requires arguments involving the denseness of contracting directions between leaves, after
a possible blow down of foliated $I$-bundles.
Using the same ideas we 
analyse how stable/unstable leaves in the universal cover intersect leaves of $\gn$,
particularly with relation to the universal circle.
We proved in \cite{Fe6} that for any pseudo-Anosov flow transverse to a 
foliation with hyperbolic leaves the following happens: given any ray in the 
intersection of a stable/leaf (in the universal cover) with  a leaf of $\gn$, then
this ray limits to a
single point in the circle at infinity of this leaf of $\gn$.
In this article
we show if $\gal$ is 
$\rrrr$-covered then given a fixed stable (or unstable leaf) and varying the leaf of $\gn$, then
these ideal points in different leaves of $\gn$
follow the identifications prescribed by the universal circle.
So clearly the universal circle is intrinsically connected with any 
regulating, transverse
pseudo-Anosov flow.
This is done in section 4.
These two results are the key tools used in the analysis of the theorem.

The next step is 
to analyse how an element of the fundamental
group acts on the universal circle. 
If an element of the fundamental group is associated to a closed orbit of
the flow, then we show that some power of it acts on the universal circle with 
a finite even number $\geq 4$ of fixed points and vice versa.
This key result depends on the analysis in section 4 and 
on further properties of the intersections of leaves
of $\gn$ and stable/unstable leaves, which is done in section 5.

Finally in section 6 we consider two pseudo-Anosov flows transverse and regulating for $\gal$.
We first prove that for each lift of a periodic orbit
of the first flow, there is a unique periodic orbit of the second flow associated
to it. 
This depends essentially on the study of group actions in section 5.
This produces a map between the orbit spaces of the two flows
restricted to lifts of closed orbits.
The final step is to show that this can be extended to an equivariant homeomorphism between
the orbit spaces.
This finishes the proof of the theorem.

\section{The case of parabolic leaves}
\label{para}

Leaves of the foliation $\gal$ are conformally either spherical, Euclidean or
hyperbolic. In this section we rule out the first case and prove the
theorem in the second case.
We say that a leaf is {\em parabolic} if it is conformally Euclidean.
These terms will be used interchangeably.

The stable and unstable foliations of $\gal$
induce $1$-dimensional perhaps singular foliations in a leaf $F$ of $\gal$.
Since there are no $1$-prongs in the stable foliation and no centers, then Euler 
characteristic disallows the existence of spherical leaves.

\begin{theorem}{}{}
Let $\gal$ be an $\rrrr$-covered foliation transverse to a pseudo-Anosov flow
$\Phi$. If $\gal$ has a parabolic leaf, then there is a compact
leaf $C$ which is parabolic 
and $M$ fibers over the circle with fiber $C$. 
In this case the flow is an Anosov flow and is a suspension flow with fiber $C$.
Therefore if an $\rrrr$-covered foliation $\gal$ has a parabolic 
leaf, then up to topological conjugacy, there is at most one pseudo-Anosov
flow transverse to $\gal$.
\end{theorem}

\begin{proof}{}
We start by proving the first statement.
If the pseudo-Anosov flow
$\Phi$ is not regulating for $\gal$ then as explained in the introduction,
the leaves of $\gal$ are Gromov hyperbolic and therefore not conformally 
Euclidean. 
Therefore $\Phi$ has to be regulating.

We assume first that $M$ is orientable.

Let $L$ be a parabolic leaf of $\gal$.

Suppose first that $\gal$ has a compact leaf. Since $\gal$ is $\rrrr$-covered,
it was shown by Goodman and Shields \cite{Go-Sh} that any compact leaf is
a fiber of $M$ over the circle.
We first want to show that there is a compact leaf which is parabolic.
This is not true in general, but it holds for $\rrrr$-covered foliations.
If the parabolic leaf
$L$ is compact we are done. Suppose then that $L$ is not compact.
We want to show that $L$ limits in a compact leaf.
In this case let $O$ be the closure of the component of the complement of
the compact leaves which contains $L$.
This component is a product of a closed surface $C$ times a closed interval
and in addition we can assume that $\gal$ is transverse
to the $I$-fibration in $O$
(see \cite{Fe2}).
Identify $C$ with the lower boundary of $O$. 
Look at the points that $L$ hits in a fixed $I$ fiber $J$. Take the
infimum of these points, call it $x$. If $x$ is in the boundary of
of $O$ we are done. The foliation in $O$ is determined by its holonomy
which is a homomorphism of $\pi_1(C)$ into the group homeomorphisms of $J$.
This holonomy has to fix $x$ for otherwise some element would bring
$x$ closer to $C$ and hence $L$ would have a point in $J$ lower than
$x$. Since the holonomy fixes $x$ then the leaf through $x$ is
compact, contrary to assumption that there are no compact leaves in
the interior of $O$.

Therefore $L$ limits on a compact leaf $C$. Since $L$ is parabolic,
then so is $C$. 
Since $\Phi$ is regulating for $\gal$ then every orbit through
$C$ intersects $C$ again, in other words $\Phi$ is a suspension flow
and the cross section is an Euclidean surface. In particular 
$\Phi$ is an Anosov flow.
Any two pseudo-Anosov flows transverse to $\gal$ will generate suspension
flows in $M$ transverse to $C$. As explained in the introduction,
any two such flows are topologically conjugate.
This finishes the analysis (in the orientable case) when there is a compact
leaf.

Suppose now that there is no compact leaf.
As proved in proposition 2.6 of \cite{Fe2} there is a unique minimal set
${\cal Z}$ in $\gal$. Since $L$ must limit in leaves in a minimal
set, then there are parabolic leaves in the minimal set, and hence
all leaves in the minimal set are parabolic.
There are at most countably many components in $M - {\cal Z}$ each
of which has a closure which is an $I$-bundle over a non compact surface.
In addition the flow can be taken to be the $I$-fibration in
this closure \cite{Fe2}.
Therefore these $I$-bundles can be blown down to leaves to yield a foliation
which is still transverse to $\Phi$ and is a minimal foliation.
Clearly this happens for any pseudo-Anosov flow transverse to $\gal$.
Therefore we may assume in this case that $\gal$ is minimal.

If all leaves of $\gal$ are planes then Rosenberg \cite{Ros} proved that
$M$ is homeomorphic to the $3$-torus and hence $\pi_1(M)$ has polynomial
growth of degree $3$. On the other hand a manifold with a pseudo-Anosov
flow has fundamental group with exponential growth \cite{Pl-Th}.
Therefore this case cannot happen.

Since $\gal$ is minimal and has a parabolic leaf, then all leaves of
$\gal$ are parabolic. Since $M$ is orientable and $\gal$ is transversely
orientable (there is a transverse flow), the leaves of $\gal$ are 
either planes, annuli or tori. We took care of the case when there
is a toral leaf and also the case when all leaves are planes. Hence there is a leaf, call if
$F$, which is an annulus. 
Since $F$ has polynomial growth, then Plante \cite{Pl} showed that
there is a holonomy invariant transverse measure supported in the
closure of $F$. Since $F$ is dense, this shows that the support of the
measure is all of $M$.

Notice that by Tischler's theorem \cite{Ti} (for the $C^0$ case by
Imanishi see \cite{Im}), then $M$ fibers over the circle and $\gal$ is
approximated arbitrarily close by fibrations.
But we need more information.

The next step is to construct an incompressible torus $T$ transverse to
$\gal$ and foliated by circles.
This will take a while.
Let $\gamma$ be a simple closed curve in $F$ which is not null 
homotopic in $F$.
Let $B$ be a small closed annulus transverse to $\gal$ and with
one boundary $\gamma$. Since there is no holonomy in $\gal$ the foliation
induced by $\gal$ in $B$ is a foliation by circles near $\gamma$ and
we may assume the other boundary leaf is also a closed curve
in a leaf of $\gal$.
Each of these circles is not null homotopic in its leaf, for otherwise
$\gamma$ would be null homotopic in $M$ contradicting Novikov's 
theorem \cite{No}.

Starting from $\gamma$, move along $F$ in a particular side (call it the
right side of $B$) until hitting $B$ again. This is possible since
leaves of $\gal$ are dense in $M$ and $\gamma$ bounds a half non compact annulus in
that side.
The first time this half annulus hits $B$ again, then it hits $B$
in a closed curve $\gamma_1$ of
the induced foliation of $\gal$ in $B$. Since $F$ is an annulus then $\gamma$ and $\gamma_1$
bound a closed annulus $A_1$ in $A_1$. We think of $\gamma, \gamma_1$ as
oriented, freely homotopic curves in $F$.
Let $B_1$ be the closed subannulus of $B$ bounded by $\gamma, \gamma_1$.
If $A_1$ approaches $B_1$ from the left side then $\gamma_1$ is freely 
homotopic to $\gamma$ in $B_1$ (since $M$ is orientable and $\gal$
transversely orientable). In this case let $T_1 = A_1 \cup B_1$,
which is a two sided torus in $M$.

In the other case $A_1$ approaches $B_1$ from the right side and then
$\gamma, \gamma_1$ are freely homotopic to the inverses of each other
in $B_1$. In this case $A_1 \cup B_1$ is a one sided Klein bottle.
In this case continue along $F$ past $\gamma_1$ until it hits
$B_1$ again in a curve $\gamma_2$ (with orientation) 
$-$ with $A_2$ the annulus in $F$ bounded by $\gamma_1, \gamma_2$.
Let $B_2$ be the subannulus of $B_1$ bounded by $\gamma, \gamma_2$ and
let $B_3$ be the subannulus of $B_1$ bounded by $\gamma_2, \gamma_1$.
If $\gamma_1, \gamma_2$ are freely homotopic in $B_1$, let
$T_1 = A_2 \cup B_3$ which is an embedded, two sided torus in $M$.
Otherwise let $T_1 = A_1 \cup A_2 \cup B_2$ which is again an embedded
two sided torus in $M$.

In any case $T_1$ is a torus obtained from an annulus $A^*$ in $F$ and
a transverse annulus contained in $B_1$ foliated by circles.
Since the annulus $A^*$ has trivial holonomy, it has a small neighborhood
which is product foliated and we can perturb $T_1$ slightly to produce
an embedded torus  $T$ transverse to $\gal$ and foliated by
circles. It follows that $T$ is incompressible but we will not need that.

Cut $M$ along $T$ to produce a manifold $M_1$ with 2 boundary tori
$U_1, U_2$ and induced $2$-dimensional foliation $\gal_1$ transverse
to the boundary of $M_1$.
Since every leaf of $\gal$ intersects $B_1$, etc.. then every 
leaf of $\gal_1$ intersects $\partial M_1$. A leaf $E$ of $\gal_1$
intersects say $U_1$ in a closed curve $\alpha$ and moving from $\alpha$ in
$E$ it has to intersect $\partial M_1$ again. Since the leaves of
$\gal$ are annuli, then all leaves of $\gal_1$ are compact annuli.
Since $M_1$ is orientable and has two boundary components, it now
follows that $M_1$ is homeomorphic to $S^1 \times V$, where $V$
is a compact annulus. Now $M$ is obtained by glueing
$U_1$ to $U_2$ preserving a circle foliation.
Hence $M$ is a nilpotent $3$-manifold. It follows that $\pi_1(M)$ has
polynomial growth, again contradicting the fact that $\pi_1(M)$ has
exponential growth \cite{Pl-Th}.
So again we conclude that this cannot happen.

We conclude that in this case $\gal$ has to have a compact leaf
$C$, which is a fiber of a fibration of $M$ over $S^1$
and $\Phi$ is topologically conjugate to a suspension.
 The result is proved in this case.
If $M$ is non orientable then it is doubly covered by an orientable manifold
and the result applies to the double cover. Hence again $\gal$ has
a compact leaf $C$ which is a fiber and the result follows in
this case as well.
This finishes the proof the theorem.
\end{proof}

\section{General facts about $\rrrr$-covered foliations}

From now on we may assume that $\gal$ has only Gromov hyperbolic leaves.
A theorem of Candel \cite{Can} then shows that there is
a metric in $M$ so that leaves of $\gal$ are hyperbolic surfaces. We assume this
is the metric we are using.
The following facts concerning $\rrrr$-covered foliations
are proved in \cite{Fe2,Cal2}.
There are two possibilities for $\gal$:

\begin{itemize}

\item
$\gal$ is {\em uniform} $-$ Given any two leaves $L, E$ of $\gn$, then they are a finite Hausdorff distance
from each other. This was defined by
Thurston  \cite{Th2}. If $a$ is the Hausdorff distance between 
the leaves $L, E$ (which depends on the pair $L, E$),
then for any $x$ in $L$ choose
$f(x)$ in $E$ so that $d(x,f(x)) \leq a$.
This map $f$ is a quasi-isometry  between $L$ and $E$ and hence induces a homeomorphism 
between the corresponding circles at infinity
$f: \pin L \rightarrow \pin E$.
Note that $f$ in general may not even be continuous. However, given the $\rrrr$-covered hypothesis,
then $f$ is boundedly well defined: any two choices of $f(x)$ are a bounded distance
from each other. The bound depends on the pair of leaves.
Clearly these identifications between circles at infinity are group equivariant under
the action by $\pi_1(M)$.
In addition they satisfy a cocycle property: given $3$ leaves $L, E, S$ of $\gn$, then
the identifications between $\pin L$ and $\pin E$ composed with those between
$\pin E$ and $\pin S$, induce the expected identifications between $\pin L$ and
$\pin S$.
Hence all circles at infinity are identified to a single circle, which is called
the {\em universal circle}
of $\gal$ or $\gn$ and is denoted by $\uu$. 
By the equivariance property, $\pi_1(M)$ acts on $\uu$.
The fact to remember here is that given
$p$ in $\pin L$ and $q$ in $\pin E$, then $p, q$ are associated to the same point of $\uu$ 
if and only if a geodesic ray $r$ in $L$ defining $p$ is a finite Hausdorff distance
in $\mi$ from a geodesic $r'$ in $E$ defining $q$.

\item
$\gal$ is not uniform. If $\gal$ is not a minimal foliation, then it has up to countably
many foliated $I$-bundles. One can collapse the $I$-bundles to produce a foliation which
is minimal (notice this does not work in the uniform case, for instance when $\gal$ is
a fibration).
If a pseudo-Anosov flow is transverse to $\gal$, then one can do the blow down so that
the flow is still transverse to the blow down foliation \cite{Fe2}.
Sometimes we will assume in this case that $\gal$ is minimal.
If $\gal$ is minimal then the following important fact is proved in \cite{Fe2}: for any
$L, E$ leaves of $\gn$, then there is a dense set of contracting
directions between them. A {\em contracting direction} is given by a geodesic $r$ in
$L$ so that the distance between $r$ and $E$ converges to $0$ as one escapes
in $r$.
Notice this only depends on the ideal point of $r$ in $\pin L$ as all such rays
are asymptotic because $L$ is the hyperbolic plane.
Any such direction produces a {\em marker} $m$. This is an embedding

$$m: [0,\infty) \times [0,1]
\ \rightarrow \ \mi$$

\noindent
 so that for each $s$ in $[0,1]$  there is a leaf $F_s$ of $\gn$ so that 

$$m([0,\infty) \times \{ s \}) \ \subset \ F_s$$

\noindent
is a parametrized geodesic ray in $F_s$. In addition for each $t$ in $[0,+\infty)$ then
$m(\{ t \} \times I)$ is a transversal to 
$\gn$ and 

$$\forall s_1, s_2 \in I, \ \ \ \ d(m(t,s_1), m(t,s_2)) \rightarrow 0 \ \ \
{\rm as} \ \ t \rightarrow \infty.$$

\noindent
Hence these geodesics of $F_{s_1}, F_{s_2}$ are 
asymptotic in $\mi$.
The contracting directions between $L, E$ induce an identification between dense sets in
$\pin L, \pin E$ which preserves the circular ordering. This extends to a homemorphism
between
$\pin L$ and $\pin E$. These homeomorphisms are clearly $\pi_1(M)$ equivariant and
in addition they satisfy the cocycle property
as in the uniform case.
Hence as before each circle at infinity is canonically identified to a fixed circle $\uu$, the
universal circle of $\gal$ or $\gn$. Finally  $\pi_1(M)$ acts on $\uu$.
\end{itemize}

We now discuss what happens if $\gal$ is not uniform and not minimal.
This was not discussed in \cite{Fe2} but it is a simple consequence
of the analysis of the minimal case as follows:
Let ${\cal Z}$ be the unique minimal set of $\gal$ \cite{Fe2}.
Blow down $\gal$ to a minimal foliation $\gal'$.
The analysis above produces the universal circle $\uu'$ for $\gal'$.
Let $\delta: M \rightarrow M$ be the blow down map sending leaves of
$\gal$ to leaves of $\gal'$ and homotopic to the identity.
Lift the homotopy to produce a lift 
$\widetilde \delta$ of $\delta$, which is a homeomorphism of $\mi$.
For any $A, B$ leaves of $\gn'$, there are $F, E$ leaves in
$\widetilde {\cal Z}$ so that $A, B$ are between $F, E$.
Let $F' = \widetilde \delta (F)$,
$E' = \widetilde \delta (E)$.
Then in $\gn'$ there is a dense set of contracting directions
between $F'$ and $E'$.
For any such there is a ray $r'$ in $F'$ asymptotic to a ray $l'$ in 
$E'$.
Under the blow up map, this produces corresponding rays in $F, E$:
a ray $r$ in $F$ which is a bounded distance from a ray $l$ in $E$.
By the $\rrrr$-covered property, the ideal point of the ray $l$ is 
the unique direction for which there is a ray a bounded distance
from $r$ in $\mi$.
This provides an identification between dense sets in $\pin F$ and
$\pin E$. This is equivariant and satisfies the cocycle property.
This can be extended to a group equivariant homeomorphism between
$\pin F$ and $\pin E$.
This produces the universal circle in this case.

Calegari \cite{Cal1} produced many examples of $\rrrr$-covered, non uniform
foliations in closed, hyperbolic $3$-manifolds.

\section{Intersections between leaves of $\gn$ and pseudo-Anosov foliations}

Let $\Phi$ be a pseudo-Anosov flow in $M^3$ closed.
Background on pseudo-Anosov flows can be found in 
\cite{Mo1,Fe6}. 
Here we always assume that there are no $1$-prong singular orbits.
The universal cover of $M$ is denoted by $\mi$.
Let $\fs, \fu$  be the stable/unstable foliations of $\Phi$ and 
$\wwp, \fns, \fnu$ the lifts to the universal cover of $\Phi, \fs, \fu$
respectively. 
Given $z$ in $\mi$ let $\ws(z)$ be the stable leaf containing $z$
and similarly define $\wu(z)$. 
Our assumption is that $\Phi$ is transverse to the foliation $\gal$ and
is regulating for $\gal$.
Therefore given any leaf $L$ of $\gn$, the foliations $\fns, \fnu$ are
transverse to $L$ and they induce $1$-dimensional singular foliations
$\fnsL, \fnuL$ in $L$.
We are in the case that leaves of $\gn$ are isometric to the hyperbolic
plane.

One key fact to be used here is that
we proved in \cite{Fe6} that 
each ray of a leaf of $\fnsL$ or $\fnuL$ accumulates in a single point of $\pin L$.
This works even if $\gal$ is not $\rrrr$-covered.

A convention that will be used throughout the article 
is the following: the group $\pi_1(M)$ acts
on several objects: the universal cover $\mi$, the orbit space $\oo$,
the universal circle $\uu$, the foliations $\fns, \fnu, \oos, \oou$, etc..
If $g$ is an element of $\pi_1(M)$ we still use the same $g$ to denote
the induced actions on all these spaces $\mi, \oo, \uu, \fns, \fnu, \oos, \oou$, etc.. 

\begin{lemma}{}{}
Suppose that a pseudo-Anosov flow $\Phi$ is regulating for an $\rrrr$-covered
foliation $\gal$. Then the stable and unstable foliations $\fns, \fnu$ have
Hausdorff leaf space. Therefore for any leaf $L$ of $\gn$, the leaves
of the one dimensional foliations $\fnsL, \fnuL$ are uniform
quasigeodesics in $L$.
\label{quasi}
\end{lemma}

\begin{proof}{}
This  is stronger than the fact that rays in these leaves limit to single
points in $\pin L$.
If we suppose on the contrary that (say) $\fns$ does not have Hausdorff
leaf space, 
then there are closed orbits
$\alpha, \beta$ of $\Phi$ (maybe with multiplicity), so that
they are freely homotopic to the inverse of each other, see \cite{Fe6}.
Lift them coherently to orbits $\widetilde \alpha, \widetilde \beta$
of $\wwp$.
Since $\Phi$ is regulating for $\gal$, then both $\widetilde \alpha$
and $\widetilde \beta$ intersect every leaf of $\gn$.

Let $g$ in $\pi_1(M)$ non trivial with $g$ leaving $\widetilde \alpha$ invariant
and sending points in $\widetilde \alpha$ forward (in terms of the flow
parameter).
Therefore $g$ acts in an increasing way in the leaf space of $\gn$.
By the free homotopy, $g$ also leaves $\widetilde \beta$ invariant
and $g$ acts decreasingly in $\widetilde \beta$, hence also in the leaf
space of $\gn$. This is a contradiction.

Hence the leaf spaces of $\fns, \fnu$ are Hausdorff. As proved in
proposition 6.11 of \cite{Fe6} this implies that for any $L$ in $\gn$,
then all leaves of
$\fnsL, \fnuL$ are uniform quasigeodesics  in $L$.
The bounds are independent of the leaf
of $\fnsL, \fnuL$ in $L$ and also of the leaf $L$ of $\gn$.
For non singular leaves, this implies that any such leaf is a bounded
distance (in the hyperbolic metric of $L$) from a minimal geodesic in $L$.
For singular $p$-prong leaves of $\fnsL, \fnuL$ the same is true
for any properly embedded copy of $\rrrr$ in such leaves.
\end{proof}

In this section we want to show that the asypmtotic behavior
of leaves of $\fnsL, \fnuL$ is coherent
with the identifications prescribed by the universal circle.

Let $\hp$ be the leaf space of $\gn$, which is homeomorphic to the set
of real numbers. Let $\pi: \mi \rightarrow M$ be the universal covering
map.

Let $r$ be ray of $\fnsL$ (or $\fnuL$) starting at a point $p$ in $L$.
In general this is {\underline {not}} a geodesic ray in $L$.
Let $E$ be any leaf of $\gn$. Since $\Phi$ is regulating, then
$\wwp_{\rrrr}(p)$ intersects $E$ and the same is true for any
point $q$ in $r$. The intersection of $\wwp_{\rrrr}(r)$ and
$E$ is a ray of $\fnsE$ $-$ again because of the regulating
condition.  This ray 
%is also quasigeodesic
%in $E$ and therefore 
also defines an unique ideal point in $\pin E$.
Since $\pin E$ is canonically identified with the universal circle
$\uu$ this defines a map 

$$f_r: \ \hp \ \ \rightarrow \ \ \uu$$

$$f_r(E) \ = \ \{ {\rm \ equivalence \ class \
in} \ \ \uu \ {\rm \ of \ the \ ideal \ point \ in} \ \ 
\pin E \ \ {\rm of \ the \ ray} \ \ 
(\wwp_{\rrrr}(r) \cap E) \ \}$$

\begin{proposition}{}{}
For any $L$ in $\gn$ and any $r$ ray of $\fnsL$ or $\fnuL$, then
$f_r: \hp \rightarrow \uu$ is a constant map.
\label{const}
\end{proposition}

\begin{proof}{}
The proof depends on whether $\gal$ is uniform or not.

\vskip .1in
\noindent
{\underline {Case 1}} $-$ $\gal$ is uniform.

\vskip .1in
\noindent
{\underline {Claim}} $-$ If $\gal$ is uniform and $\Phi$ is transverse and regulating for $\gal$, 
then for any $S, E$ leaves of $\gn$, there is
a bound on the length of flow lines from $S$ to $E$.
The bound depends on the pair $S, E$.

Otherwise we find $p_i$ in $S$ with $\wwp_{t_i}(p_i)$ in $E$ and $t_i$ converging to
(say) infinity.
Up to subsequence assume that $\pi(p_i)$ converges to a point $p$ in $M$.
Take covering translations $g_i$ in $\pi_1(M)$ with $g_i(p_i)$ converging
to $p_0$. For each $i$ take $q_i$ in $g_i(E)$ with $d(q_i, g_i(p_i)) < a$ for
fixed $a$. This uses the uniform property. Up to subsequence assume that
$q_i$ converges and hence $g_i(E)$ converges to a leaf $E_0$.
The orbit of $\wwp$ through $p_0$ intersects $E_0$, since the flow is regulating.
Hence there is $t_0$ with $\wwp_{t_0}(p_0)$ in $E_0$.
By continuity of flow lines of $\wwp$, then for any $z$ in $\mi$ near $p_0$ 
and $G$ leaf of $\gn$ near $E_0$, then there
is $t$ near $t_0$ so that $\wwp_t(z)$ is in $G$. But
$\wwp_{t_i}(g_i(p_i))$ is in $g_i(E)$, which is a leaf near $E_0$ and
$t_i$ converges to infinity, contradiction.
This proves the claim.
Notice that it  is not necessary for $\Phi$ to be pseudo-Anosov in this claim,
just that it is regulating.

Let $l$ be the geodesic ray in $L$ with starting point $p$ and a finite Hausdorff
distance (in $L$) from $r$.
By the above $\wwp_{\rrrr}(r)$ intersects $E$ in a ray $r'$ of $\fnsE$ which is
a bounded distance from $r$ in $\mi$.
The ray $r'$ is also a uniform quasigeodesic ray in $E$, hence
 $r'$ is a bounded distance in $E$ from a geodesic ray $l'$.
Then $l, l'$ are a finite distance from each other in $\mi$.
%from the geodesic ray in $E$ which defines the point identified to
%the ideal point of $r$.
The definition of the universal circle in the uniform case implies that
$r, r'$ define the same point in $\uu$. This establishes this case.

\vskip .2in
\noindent
{\underline {Case 2}} $-$ $\gal$ is not uniform.

In this case, first assume that $\gal$ is minimal. Therefore between any two leaves of
$\gn$, there is a dense set of contracting directions.
The proof will be done by contradiction.
Let $r$ be a ray of a leaf of $\fnsL$ for some $L$ in $\gn$ with initial
point $p$. Let $a$ be the ideal point of $r$ in $\pin L$. Suppose that
for some $E$ leaf of $\gn$, then 

$$r' \ = \ \wwp_{\rrrr}(r) \cap E \ \ \ {\rm defines \
a \ distinct \ point \ in} \ \ \  \uu$$

\noindent
Let $b$ be the point in $\pin L$ identified to the ideal point of $r'$ in $\pin E$, by
the universal circle identification. Hence $a, b$ are different.
By density of contracting directions between $L$ and $E$, there are
points $c, d$ in $\pin L$ which separate $a$ from $b$ in $\pin L$ and so
that $c, d$ correspond to contracting directions between $L$ and $E$.
Let $m_1, m_2$ be markers between $L$ and $E$ associated to the
contracting directions $c, d$ respectively.
Let $B_i = Image(m_i)$ and let $C$ be the union of the points in $\mi$ contained
in leaves intersecting the markers $m_1, m_2$. 
Removing initial pieces if necessary we may assume that $B_1, B_2$ are disjoint.
Since  $m_i( \{ t \} \times I)$ is a very small transverse arc if $t$ is big enough,
we can also assume the following: if $z$ is in $B_1$ or $B_2$ then 
$\wwp_{\rrrr}(z)$ will intersect any leaf $S$ in $C$ near $z$, producing
a small transversal from $L$ to $E$ passing through $z$.
For each leaf $S$ of $\gn$ intersected by the markers, let 

$$r_S \ \ = \ \ {\rm geodesic \ arc \ in} \ \ S \ \ {\rm joining \ the \ endpoints \ of}
\ \ Image(m_1) \cap S \ \ {\rm and} \ \  Image(m_2) \cap S.$$

\noindent
Let $A$ be the union of the $r_S$ for such $S$. This is topologically
a rectangle with the bottom in $L$ the top in $E$ and the sides transversals
from $L$ to $E$.
Then $A \cup B_1 \cup B_2$ separates $C$
into $2$ components $C_1, C_2$.
Since $\{ a, b \}$ is disjoint from $\{ c, d \}$ the ray $r$ does not accumulate
on $c$ or $d$ in $\pin L$. Hence
starting with a smaller ray $r$ if necessary we may assume also that $r, r'$ are
disjoint from $B_i$ and far away from it.
In particular the flow line through any point of $r$ will not intersect
$B_i$, since points in $B_i$ are in very short transversals from $L$ to $E$.

By renaming $C_1, C_2$ we may assume that $r$ is in $C_1$ and $r'$ is in $C_2$.
For each $z$ in $r$ it is in $C_1$,
then the flow line through $z$ intersects $E$ in $r'$ which is in $C_2$.
Therefore this flow line has to intersect $A \cup B_1 \cup B_2$.
The above remarks imply that this flow line cannot intersect either
$B_1$ or $B_2$.
Hence this flow line must intersect $A$.
Since $A$ is compact we can choose 
$z_i$ in $r$ escaping in $r$ so that $\wwp_{\rrrr}(z_i)$ intersects
$A$ in 

$$q_i \ = \ \wwp_{t_i}(z_i) \ \ \ {\rm and} \ \ q_i \rightarrow q \in A$$

\noindent
Since $z_i$ escapes in $r$, it follows that $t_i$ converges to infinity.
By the regulating property of $\Phi$, the orbit through $q$ intersects $L$.
Hence nearby
orbits intersect $L$  in bounded time, contradicting that $t_i$ converges to infinity.

This contradiction shows that $r'$ has to define the same point in $\uu$ that $r$ does.
This finishes the proof when $\gal$ is minimal.

If $\gal$ is not minimal, then first blow down $\gal$ to a minimal foliation $\gal'$.
We can assume that $\Phi$ is still transverse to $\gal'$.
Now use the proof for $\gal'$ as above.
The walls $A \cup B_1 \cup B_2$ for $\gn'$ pull  back to walls for $\gn$.
Because the foliation $\gal$ is a blow up of $\gal'$ and $\Phi$ is 
transverse to both of them, it follows that
flowlines of $\wwp$ cannot cross the two ends of the pullback walls
and if necessary
can only cross the compact part of these walls.
Therefore the same arguments as above prove the result in this case.
This finishes the proof of the proposition.
\end{proof}

For a leaf $F$ of $\gn$ we consider $F \cup \pin F$ as the canonical compactification
of $F$ as a hyperbolic plane.
Given any two leaves $F, E$ in $\gn$, then using the universal circle analysis
there is a homeomorphism between $\pin F$ and $\pin E$.
In addition if a flow $\Phi$ is regulating for $\gal$ then
there is also a homeomorphism between $F, E$ by moving along flow lines.
We next show that these are compatible:

\begin{proposition}{}{}
Given $F, E$ in $\gn$ consider the map $g$ from $F \cup \pin F$ to
$E \cup \pin E$ defined by: if $x$ is in $F$ then move along the
flow line of $\wwp$ through $x$ until it hits $E$. The intersection
point is $g(x)$.
If $x$ is in $\pin F$, let $g(x)$ be the point in $\pin E$ associated
to $x$ by the universal circle identification.
Then $g$ is a homeomorphism.
In addition these homeomorphisms are  group equivariant and satisfy the
cocycle condition.
\end{proposition}

\begin{proof}{}
The map $g$ is a bijection. We only need to show that it is continuous,
since the inverse is a map of the same type. The equivariance and cocycle properties
follow immediately from the same properties for flowlines and identifications
induced by the universal circle.

We now prove continuity of $g$:
This is very similar to the previous proposition and we will use the setup of
that proposition.
The first possibility is that $\gal$ is uniform.
Then as seen in the previous proposition the map $g: F \rightarrow E$ is a
quasi-isometry and it induces a homeomorphism $g^*$ from $F \cup \pin F$
to $E \cup \pin E$. The image of an ideal point $p$ in $\pin F$ is determined
by the ideal point of $g(r)$ where $r$ is a geodesic ray in $F$ with ideal
point $p$. But $g(r)$ is a bounded distance from $r$ in $\mi$ and
this is exactly the identification associated to the
universal circle.

Suppose now that $\gal$ is not uniform.
Suppose first that $\gal$ is minimal. 
We know that $g$ restricted to both $F$ and $\pin F$ are homeomorphisms.
Since $F$ is open in $F \cup \pin F$ all we need to do is to show that $g$
is continuous in $\pin F$.
Let $a$ in $\pin F$ and $(a_i)$ converging to $a$ in $F \cup \pin F$, so 
we may assume that $a_i$ is in $F$.
Suppose by way of contradiction that $g(a_i)$ converges to $g(b)$ where
$b$ is not $a$.
Choose $c, d$ in $\pin F$ which separate $a, b$ in $\pin F$. Then construct
the wall $A \cup B_1 \cup B_2$ as in the previous proposition.
The flow lines from $a_i$ to $g(a_i)$ have to intersect this wall in
a compact set, contradiction as in the previous proposition.
This finishes the proof if $\gal$ is minimal.

If $\gal$ is not minimal, then use the same arguments as in the end of
the previous proposition to deal with this case.
\end{proof}

%Notice that $g$ restricted to both $F$
%is a homeomorphism to $E$ and likewise for the restriction to $\pin F$.
%Since $F$ is open in $F \cup \pin F$, it follows that $g$ is continuous
%in $F$.
%Let then $p$ in $\pin F$ and $p_i$ converging to $p$ in $F \cup \pin F$.
%As $g$ restricted to $\pin F$ is also a homeomorphism, we may assume 
%that $p_i$ are in $F$.
%
%Let $\{ r_n \ | \ n \in {\bf N} \}$ be a sequence of leaves (say) in
%$\fnsF$ which defines a neighborhood system of $p$ in $F \cup \pin F$.
%For each $n$, $r_n$ is contained in leaf $R_n$ of $\fns$.
%For simplicity assume that $r_n$ are all non singular.
%Given $n$, then for $i$ big enough, the $p_i$ is the component of
%$F - r_n$ which limits on $p$ in $F \cup \pin F$.
%Let $\partial r_n = \{ a_n, b_n \}$ which are points in $\pin F$ distinct
%from $p$.
%Flow $r_n$ to $r'_n$ in $E$, that  is $r'_n = R_n \cap E$.
%Let $\partial r'_n =  \{ a'_n, b'_n \}$ two points in  $\pin E$.
%By lemma \ref{const} it follows that
%$a'_n = g(a_n), \ b'_n = g(b_n)$. In addition $g(p_n)$ is in the component
%of $E - r'_n$ which limits on $g(p)$.
%
%
%As the $(r_n)$ define a neighborhood system of $p$
%in $F \cup \pin F$, then the sequences $(a_n), (b_n)$ in $\pin F$ converge
%to $p$.
%As $g$ restricted to $\pin F$ is a homeomorphism then 
%%the sequences $(a'_n), (b'_n)$ converge to $g(p)$. Since $r'_n$ are
%uniform quasigeodesics in $E$, with these ideal points, then 
%$r'_n$ converge to $g(p)$ in $E \cup \pin E$. This now implies
%that $g(p_n)$ converges to $g(p)$ and $g$ is continuous at
%$p$.  This finishes the proof of the proposition.
%\end{proof}

This proposition allows us to put a topology in $\oo \cup \uu$
as follows: Consider any leaf $L$ of $\gn$. There are homeomorphisms
between $L$ and $\oo$ and $\pin L$ and $\uu$.
The combined map induces a topology in $\oo \cup \uu$.
The previous proposition shows that this topology is
independent of the leaf $L$ we start with.
In addition
covering translations induce homeomorphisms of $\oo \cup \uu$ $-$ this is 
because if $L$ is in $\gn$ and $g$ in $\pi_1(M)$ then $g$ is 
a homeomorphism from $L \cup \pin L$ to \ $(g(L) \cup \pin g(L))$, \
both of which are homeomorphic to $\oo \cup \uu$.
We think of this as an action on $\oo \cup \uu$.
Given $g$ in $\pi_1(M)$, then the notation $g$ will also 
denote the induced map in $\oo \cup \uu$.
% and $L \cup \pin L$.
%(We have not defined a 
%The topology in the union $\oo \cup \uu$ is defined to be the topology
%induced from the identification with $L \cup \pin L$.
The analysis above makes it clear that $g$ in $\pi_1(M)$ acts
as an orientation preserving way on $\oo$ if and only if it acts
as an orientation preserving way on $\uu$.

\section{Action of elements of $\pi_1(M)$}

The main purpose of this section is to analyse 
how elements of $\pi_1(M)$ act on $\uu$.
We first need a couple of auxiliary results.
Here is some notation/terminology which will be used in the sequel.
Let $\oo$ be the orbit space of the lifted flow $\wwp$. The space $\oo$ is always homeomorphic
to the plane \cite{Fe-Mo}. 
The foliations $\fns, \fnu$ induce $1$-dimensional, possibly singular
foliations $\oos, \oou$ in $\oo$.
The only possible singularities are of $p$-prong type.
A point $x$ in $\oo$ is called {\em periodic} if there is $g \not = id$ in $\pi_1(M)$ with
$g(x) = x$. Let

$$\Theta: \ \mi \ \rightarrow \ \oo \ \ \ \ {\rm be \ the \ projection \ map}$$

\noindent
An orbit $\alpha$ of $\wwp$ is periodic if $\Theta(\alpha)$ is periodic.
A {\em line leaf} of $\fnsL$
is a properly embedded copy of $\rrrr$ in a leaf of $\fnsL$ 
of a leaf $L$ of $\gn$ so that: if $l$ is in a singular leaf $r$ of $\fnsL$, then $r - l$ does not
have prongs of $r - l$ on both sides of $l$ in $L$. A singular leaf with a p-prong singularity
has $p$ lines leaves. 
Consecutive line leaves intersect in a ray of $\fnsL$.
Non singular leaves are line leaves themselves.
Similarly one defines line leaves for $\fnuL, \oos, \oou, \fns, \fnu$.
Given $z$ in $\mi$ let $\ws(z)$ be the stable leaf containing $z$.
The {\em sectors} of $\ws(z)$ are the connected components of $\mi - \ws(z)$.

\begin{lemma}{}{}
Let $\Phi$ be regulating for $\gal$ which is $\rrrr$-covered with
hyperbolic leaves.
Let $l_i$ be line leaves of $\fnsLi$ where $L_i$ are leaves of $\gn$.
Suppose that there are $p_i$ in $l_i$ so that $p_i$ converges in $\mi$ to
a point $p$  in a leaf $L$ of $\gn$.
If $\ws(p)$ is singular assume that all $p_i$ are in the closure of a sector
of $\ws(p)$.
Then there is a line leaf $l$ of $\fnsL$ with $p$ in $l$ and $l_i$ converging
to $l$ in the geometric topology of $\mi$.
In addition if
$s_i$ are the geodesics in $L_i$ a bounded distance from $l_i$ in $L_i$ 
and $s$ is the geodesic a bounded distance from $l$ in $L$ then
$s_i$ converges to $s$ in the geometric topology of $\mi$.
\label{break}
\end{lemma}

\begin{proof}{}
We first prove the statement about $l_i$ and $l$.
Geometric convergence means that if $z$ is in $l$ then there are $z_i$ in $l_i$
with the sequence $(z_i)$ converging to $z$ and in addition if $z_{i_k}$ is in $l_{i_k}$
and $(z_{i_k})$ converges to $w$ in $\mi$ then $w$ is in $l$.

Since the flow $\Phi$ is regulating for
$\gal$, then $l_i$ flows into line leaves 
$r_i$ of $\fnsL$.
The points $p_i$ flow to $q_i$ in $L$ and clearly $q_i$ converges to $p$.
Hence there is a line leaf $l$ of $\fnsL$ through $p$, so that any point 
$z$ in $l$ is the limit of a sequence $(z'_i)$ with $z'_i$ in $r_i$.
If $\ws(p)$ is singular, this uses the fact that the $p_i$ are all in
the closure of a sector of  $\ws(p)$. Otherwise it could easily be
that different subsequences of $r_i$ converge to distinct line
leaves of $\fnsL$.
%The piece of of $r_i$ from $q_i$ to $z_i$ is very close to $l$ and
%hence there are
%$z_i$ in $l_i$ with $(z_i)$ converging to $z$. 
Let $z_i$ in the $z'_i$ orbit with $z_i$ in $l_i$.
Then $z_i$ converges to $z$. 
This shows that any
$z$ in $l$ is the limit of a sequence in $l_i$.

Now suppose that $(z_{i_k})$ is a sequence converging to $z$ with
$z_{i_k}$ in $L_{i_k}$.
Here $p_{i_k}$ is in $L_{i_k}$ and $p$ is in $L$ and hence $L_{i_k}$ converges
to $L$ in the leaf space $\hp$ of $\gn$. Since $\hp$ is Hausdorff then 
no sequence of points in $L_{i_k}$ converges to a point in another leaf of $\gn$.
If follows that $z$ is in $L$.
Let

$$V_k \ = \ \ws(p_{i_k}), \ \ \ \ \ \ V \ = \ \ws(p)$$

\noindent
Then $V_k$ converges to $V$. 
By lemma \ref{quasi} the leaf space of $\fns$ is also Hausdorff.
It follows that $z$ is in $V$. Hence $z$ is in $L \cap V = \tau$. 
It was also proved in \cite{Fe6} that $L \cap V$ is connected
and hence $\tau$ is exactly the leaf of $\fnsL$ containing $p$.

If $\tau$ is non singular this finishes the proof of the first statement.
Suppose then that $\ws(p)$ is singular.
Since the $p_i$ are in the closure of a sector of $\ws(p)$ then so are the
$l_{i_k}$ and hence the
$z_{i_k}$. Consequently the same is true of $z$.
The boundary of this sector is a line leaf of $\ws(p)$ and so $z$ is in the
corresponding line leaf of $\fnsL$, which is $l$.
This finishes the proof of the first statement of the lemma.

\vskip .1in
We now consider the second part of the lemma.
By lemma \ref{quasi} 
the leaves of $\fnsE$ are uniform quasigeodesics in $E$ for any $E$ leaf of $\gn$.
Let then $b > 0$ so that any line leaf of $\fnsE$ is $\leq b$ from the corresponding
geodesic in $E$ and likewise for arcs in such leaves.
Let $l_i$ be line leaves of $\fnsLi$, $l$ its limit in a leaf $L$ of $\gn$
as in the first part of the lemma. 
Let $s_i$ be 
the geodesics in $L_i$ corresponding to $l_i$ and let $s$ the geodesic in $L$
corresponding to $l$.

For any $\epsilon > 0$ there is fixed $f(\epsilon) > 0$ so that if two geodesic segments in
the hyperbolic plane have length bigger than \ $3 f(\epsilon)$ \ and corresponding endpoints
are less than $2b + 2$ from each other, then except for segments of length
$f(\epsilon)$ adjacent to the endpoints, then the rest of the segments are 
less than $\epsilon/3$ from each other.

Let then $z$ in $s$.  Given $\epsilon > 0$, find $w', u'$ in $s$ which are 
$(3f(\epsilon) + 2b + 1)$ distant
from $z$.
There are $w, u$ in $l$ with 

$$d_L(w,w') \ < \ b + \frac{1}{2}, \ \ \ \ d_L(u,u') \ < \ b + \frac{1}{2}$$

\noindent
Let $\tau$ be the segment of $l$ between $w, u$.
There is a corresponding segment of $\tau_i$ of $l_i$ between points $w_i, u_i$ so that 
the Hausdorff distance in $\mi$ from $\tau$ to $\tau_i$ is $< < 1$.
The corresponding geodesic segment $m_i$ from
$w_i$ to $u_i$  in $L_i$ is less than $b$ from $\tau_i$ and 
by choice of $w', u'$ then the midpoint 
of $m_i$
is less than $\epsilon /3$ from a point $v_i$ in $s_i$. 
%In the same way $m_1$ is less than $\epsilon/3$
%from a point $v_i$ in $s_i$. 
Hence $v_i$ is less than $\epsilon$ from $z$.
By adjusting the $\epsilon$ to converge to $0$ and 
the $i$ to increase, one finds $v_i$ in $s_i$ with $v_i$ converging to $z$.

Suppose now that $z_{i_k}$ are in $s_{i_k}$ with $s_{i_k}$ 
contained in $L_{i_k}$. Suppose the sequence $s_{i_k}$ 
 converges to $z$ in $\mi$.
The proof is very similar to the above:
Fix $\epsilon > 0$.
Choose big segments in $s_{i_k}$ centered in $z_{i_k}$.
The length is fixed and depends on $\epsilon$.
There are geodesic arcs of $L_{i_k}$ with endpoints in 
the leaves $l_{i_k}$ whose midpoints  are very close
to $z_{i_k}$. Very close depends on $\epsilon$ and the length above.
There are arcs in $l_i$ with these endpoints so that the
above arcs converge up to a subsequence
to a segment in $l$ by the first part of the lemma.
The geodesic arcs above 
converge to a geodesic arc with endpoints in $l$.
Up to subsequence the midpoints of the geodesic arcs (that is, the $z_{i_k}$)
converge to a point (this point is $z$)
which
is close to a point in $s$, closeness depending on $\epsilon$. 
Now make $\epsilon$
converge to $0$ and prove that $z$ is in $l$.
This finishes the proof of the lemma.
\end{proof}

At this point it is 
convenient to do the following: for the remainder of the article we fix a leaf
$L$ of $\gn$.
% and identify $\uu$ with $\pin L$.
%Since each orbit of $\wwp$ intersects $L$ once, we identify the orbit 
%space $\oo$ of $\wwp$ with the leaf $L$. An orbit of $\wwp$ is associated to its intersection
%with $L$.
%This identification is a homeomorphism.
The bijection $L \cup \pin L \rightarrow \oo \cup \uu$ is a homeomorphism.
Therefore 
the action of $\pi_1(M)$ on $\oo \cup \uu$ induces an action by homeomorphisms on 
$L \cup \pin L$ under this identification.
%$\oo$. Under the identification of $\oo$ with $L$ it induces an action on $L$.
This action leaves invariant the foliations $\fnsL, \fnuL$ $-$ which are the intersections
of $\fns, \fnu$ with $L$.
%As seen  in proposition \ref{compat} there is also an action on 
%$\oo \cup \uu$ and hence an action by homeomorphisms on
%$L \cup \pin L$.

We need one more auxiliary fact.

\begin{lemma}{}{}
Let $E$ be a leaf of $\gn$ and $l_1, l_2$ distinct leaves of $\fnsE$ or $\fnuE$.
Then $l_1, l_2$ do not share an ideal point in $\pin E$.
\label{brick}
\end{lemma}

\begin{proof}{}
%By the identification between circles at infinity, we may assume that 
%$E$ is the fixed leaf $L$.
Suppose first by way of contradiction that there are $l_1, l_2$ rays in 
leaves of $\fnsE$ for some $E$ in $\gn$
with the same ideal point $a$ in $\pin E$ and so that $l_1, l_2$ do not share a subray. 
We can assume that $l_1, l_2$ do not have singularities.
Let $u_j, j = 1,2$ be the starting point of $l_j$.
Let $r_j, j = 1,2$ be a line leaf of $\fnsE$ containing $l_j$.
Choose points $p_i$ in $l_1$ escaping in $l_1$. 
As explained before the leaves of $\fnsE$ are uniform quasigeodesics in $E$ and
hence they are at a bounded distance in $E$ from geodesics in $E$. This implies that
there are $q_i$ in $l_2$ so that $q_i$ are a bounded distance from $p_i$ in $E$.
Up to taking a subsequence we may assume that $\pi(p_i)$ converges in $M$.
Let then $g_i$ in $\pi_1(M)$ with $g_i(p_i)$ converging to $p_0$. 
For simplicity of explanation we assume that the leaf of $\gn$ containing
$p_0$ is the fixed leaf $L$ as above.
Let $v_1$ be the line leaf of $\fnsL$
containing $p_0$ and which is the limit of the $g_i(r_1)$ as proved in the
previous lemma.
If $\ws(p_0)$ is singular then,
up to taking a subsequence, we may assume that the $g_i(p_i), g_i(r_i)$ satisfy
the requirements of the previous lemma.
%Notice that $g_i(l_1)$ flow entirely into $L$ by the regulating property.

Since the distance along $g_i(E)$ from $g_i(p_i)$ to $g_i(q_i)$ is bounded
we may assume up to subsequence that $g_i(q_i)$ also converges and let $q_0$
be its limit. It follows that $q_0$ is also in $L$ and let $v_2$ be the line leaf
of $\fnsL$ containing $q_0$ which is the limit of $g_i(r_2)$. 
Here
the rays 

$$g_i(l_1), \ g_i(l_2) \ \ {\rm in} \ \  E \ \ \ {\rm have \ the \ same \
ideal \ point} \ \   g_i(a) \ \ {\rm in}  \ \ \pin (g_i(E))$$

\noindent
The line leaves $r_j$ are uniform quasigeodesics in $E$ and a bounded distance
from a geodesic $s_j$ in $E$.
Hence the geodesics $g_i(s_1), g_i(s_2)$  share an ideal point in $\pin g_i(E)$. 
By the second part of the previous lemma $g_i(s_j)$ converges to a geodesic $t_j$ in $L$
with same ideal points as $v_j$ for both
$j = 1,2$.
By continuity of geodesics in leaves of $\gn$, it follows that $t_1$ and $t_2$
share an ideal point.
%The second part of the previous lemma then implies that $v_1, v_2$ share an ideal point
Therefore $v_1, v_2$ share an ideal point in $\pin L$.
%By the construction of the universal circle in either the uniform or non
%uniform situation it follows that $v_1, v_2$ share an ideal point
%in the direction corresponding to the rays $g_i(l_1), g_i(l_2)$.
%Notice that $v_1, v_2$ are quasigeodesics in $L$.

We claim that $v_1, v_2$ also share the other ideal point.
%For each $i$ choose $p'_i$ in $g_i(v_1)$ which is a bounded distance from
%$g_i(p_i)$. 
The line leaves $g_i(r_1), g_i(r_2)$ have big segments from 

$$g_i(u_1) \ \ {\rm to} \ \ g_i(p_k) \ \ \ \ {\rm and} \ \ \ \ 
g_i(u_2) \ \ {\rm to} \ \ g_i(q_k)$$

\noindent
which are boundedly close to each other. Here $k > > i$ and so $g_i(p_i)$ is in
these segments.
Also $g_i(p_i)$ converges to $p_0$.
The corresponding geodesic arcs between the points above have endpoints which are boundedly
close to each other.
As explained in the proof of the previous lemma
they have middle thirds which are arbitrarily close to each other.
The limits of the geodesic arcs
are contained in $t_1$ and $t_2$. This shows that $t_1$ and $t_2$ have points
in common and therefore are the same geodesic.
%Now $g_i(s_1)$ and $g_i(s_2)$ are geodesics in $g_i(E)$ which 
%are asymptotic and the have points $p'_i$ converging to a point $p'_0$ in $S$
%and also points boundedly close to $g_i(u_1)$ and $g_i(u_2)$. It follows that
%the distance from $p"_i$ to $g_i(r_2)$ converges to $0$. This implies

\vskip .1in
Suppose first that $v_1, v_2$ are distinct.
The two line leaves $v_1, v_2$ of $\fnsL$ have the same 
two ideal points, which we denote by $a_1, a_2$.
The line leaves 

$$v_1, \ v_2 \ \ \ {\rm bound \ a \ region} \ \  R \ \ {\rm in} \ \  L$$

\noindent
For any stable leaf $l$  of
$\fnsL$ in $R$ then $l$ has ideal points which can only be $a_1, a_2$.
But $l$ is a quasigeodesic in $L$. Therefore this leaf is non singular
and has ideal points exactly $a_1, a_2$. Now consider a periodic orbit
$\alpha$ of $\wwp$ intersecting $L$ in $R$ very close to $v_1$ so that the unstable leaf
$\wu(\alpha)$ intersects $v_1$. 
Notice that the set of periodic orbits of $\Phi$ is dense in $M$ 
when $\Phi$ is transitive as proved by Mosher  \cite{Mo1}.
In addition if $M$ is atoroidal then $\Phi$ is transitive
\cite{Mo1}.

We now use that $L \cup \pin L$ is identified with $\oo \cup \uu$.
Let $g$ in $\pi_1(M)$ non trivial so that $g(\alpha) = \alpha$ and in addition
$g$ leaves invariant all components of $\ws(\alpha) - \alpha$.
Under the identifications above 
then 

$$g \ \ \ {\rm  fixes} \ \ \  a_1 \ {\rm and} \ a_2 \ \ {\rm in} \ \ \pin L$$

\noindent
Notice that $a_1, a_2$ are the ideal points of $\ws(\alpha) \cap L$ in
$\pin L$.
Assume that $g^n(\ws(v_1))$ moves away from $\ws(\alpha)$
when $n$ converges to infinity.
Since $v_1$ (line leaf of $\fnsL$) has ideal points $a_1, a_2$, it follows that
the same happens for all leaves $g^n(\ws(v_1)) \cap L$. These line leaves are nested
in $L$ 
and they are uniform quasigeodesics in $L$, so they cannot escape compact sets in $L$.
Hence they have to limit in a line leaf $v$ of $\fnsL$. Since the leaf
space of $\fnsL$ is Hausdorff, the limit is unique, which implies that
$g(v) = v$.
The leaf $z$ of $\oos$ corresponding to $v$ is also invariant under $g$.
This produces a point $y$ of $\oo$ in $v$ which is invariant under $g$.
Let $\beta$ be the orbit of $\wwp$ with $\Theta(\beta) = y$.
But also $g$ leaves invariant the point $x = \Theta(\alpha)$.
This shows that there
are 2 fixed points in $\oo$ under $g$.
% (these are $x$ and $\Theta(\alpha)$.
Then $\pi(\alpha), \pi(\beta)$ are closed orbits of $\wwp$ which
up to powers are freely homotopic to the inverse of each other.
Since $\Phi$ is regulating, this is impossible:
%Basically one proves first that the above conditions imply that
%there is an orbit $\gamma$ of $\wwp$ so that
%$\pi(\gamma)$ is freely homotopic to the inverse of $\pi(\alpha)$ \ \cite{Fe1}.
Notice that $g$ is associated to the negative flow direction in $\alpha$
$-$ as it acts as an expansion in the set of orbits of
$\wu(\alpha)$.
The regulating property applied to $\alpha$ implies that $g$ acts freely
and in an decreasing fashion on the leaf space $\hp$ of $\gn$.
The property that $\pi(\beta)$ is freely homotopic to the inverse
of $\pi(\alpha)$ implies that $g$ would have to act in a decreasing way
on $\hp$, contradiction.
Notice that the last argument is about the leaf space of $\gn$ and not of $\fns$.
This contradiction
shows that $l_1, l_2$ cannot have the same ideal point in $E$.
This finishes the analysis if $v_1, v_2$ are distinct.

If $v_1 = v_2$, then for $i$ big enough we may assume that
$p_i$ is very closed to $q_i$. Then one can choose
$\alpha$ periodic with $\wu(\alpha)$ intersecting both $l_1$ and 
$l_2$. It follows that $\ws(\alpha) \cap L$ has one endpoint $a$.
Then one applies the same arguments as in the case $v_1, v_2$ distinct
to produce a contradiction.
This finishes the first part of the lemma.

\vskip .1in
We now prove that if $l_1$ is a ray in a  leaf of $\fnsL$ and $l_2$ is 
ray in a leaf of $\fnuL$ then they cannot share an ideal point in $\pin L$.
Suppose this is not the case. Apply the same limiting procedure as above
to produce a stable line leaf $s_1$ in $\fnsL$ and an unstable line leaf
$s_2$ in $\fnuL$ which share two ideal points. 
Clearly in this case they cannot be the same leaf and
they bound a
region $R$ in $L$ with ideal points $a_1, a_2$. Consider a non singular stable leaf $l$ 
intersecting
$s_2$. Then it enters $R$ and cannot intersect the boundary of $R$ (in $L$) again.
Therefore it has to limit in either $a_1$ or $a_2$ and share an ideal
point with a ray of $s_1$. This is disallowed by the first part of the proof.
\end{proof}

Given these facts the following happens:
For any $L$ in $\gn$ and leaf $l$ in $\fnsL$ if $l$ is non singular
let $l^*$ be the geodesic in $L$ with same ideal points as $l$.
If $l$ is a $p$-prong leaf, let $\delta_1,...,\delta_p$ be the
line leaves of $l$ and $\delta^*_i$ be the corresponding geodesics.
In this case let $l^*$ be the union of the $\delta^*_i$, which is
a $p$-sided ideal polygon in $L$.
Let $\lsl$ be the union of such $l^*$ for $l$ in $\fnsL$ and
similarly define $\lul$.

By lemma \ref{break} it follows that $\lsl, \lul$ are closed subsets
of $L$ and there are geodesic laminations
in $L$.
The complementary regions of $\lsl$ are exactly those associated
to $p$-prong leaves of $\fnsL$ this also follows from lemma \ref{break}
and hence are finite sided ideal polygons.
As leaves of $\fnsL$ are uniform quasigeodesics (lemma \ref{quasi}), then $\lsl$ varies
continuously if $L$ varies in $\gn$.
This produces a lamination in $M$ which intersects leaves of 
$\gal$ in geodesic laminations. As $\fnsL, \fnuL$ have no rays
which share an ideal point, it follows that $\lsl$ is transverse
to $\lul$.
It now follows that for any $p$ in $\pin L$, then $p$ has a neighborhood
system in $L \cup \pin L$ defined by a sequence of leaves in
either $\lsl$ or $\lul$.
herefore the same holds for $\fnsL, \fnuL$ as these are uniform
quasigeodesics.

We now analyse the properties of the action of $\pi_1(M)$ on $\uu$.

\begin{proposition}{}{}
Let $\gal$ be an $\rrrr$-covered foliation with a transverse regulating
pseudo-Anosov flow $\Phi$.
Let $g$ in $\pi_1(M)$ be a non trivial element. Then one of the
following options must happen:

\vskip .08in
$I$ $-$ If $g$ fixes $3$ or more points in $\uu$, then $g$ does not act
freely on $\oo$ and has a unique fixed point $x$ in $\oo$. Here $g$
is associated to a closed orbit of $\Phi$.
In addition $g$ acts by an orientation preserving homeomorphism of $\oo$
and $g$ leaves invariant each prong of $\oos(x), \oou(x)$
when acting  on $\oo$. Hence $g$ fixes
the ideal points of $\oos(x), \oou(x)$ in $\uu$ which are even in number.
These are the only fixed points of $g$ in $\uu$ and they are alternatively
repelling and attracting;

\vskip .08in
$II$ $-$ $g$ fixes exactly two points in $\uu$. Then, either \ 1) $g$ acts freely on
$\oo$ and there is one attracting and one repelling fixed point 
in $\uu$; or
\ 2) $g$ fixes a point $x$ in $\oo$ and leaves invariant exactly two prongs 
of (say) $\oos(x)$
but not those of $\oou(x)$ or any other possible prongs of $\oos(x)$ (or vice versa). 
Here $g$ reverses orientation in $\oo$.
The orbit associated to $x$ may be non singular in which case all prongs
of $\oos(x)$ are left invariant and there are 4 fixed points in $\uu$ under the
square of $g$. The orbit associated to $x$ may be singular. Then the
square of $g$ has more than $4$ fixed points in $\uu$.

\vskip .08in
$III$ $-$ $g$ has no fixed point in $\uu$. Then $g$ fixes a single point
$x$ in $\oo$ and a power of $g$ fixes an even number $\geq 4$ of points
in $\uu$.

\vskip .08in
Consequently, $g$ always fixes a finite even number of points in $\uu$ (it may
be zero).
\label{act}
\end{proposition}

\begin{proof}{}
%Suppose first that $g$ fixes $3$ or more points in $\uu$.
Since $g$ acts on $\oo$ and leaves invariant the foliation $\oos$, then 
it acts on the leaf space $\hhs$ of $\oos$.
This is the same as the leaf space of $\fnsL$ (under the identification
of $\oo$ with $L$).
%The leaves of $\fnsL$ are uniform quasigeodesics in $L$ (maybe with prongs).
%If the leaf space of $\fnsL$ were not Hausdorff then one would find
%$p_n, q_n$ in leaves $l_n$ of $\fnsL$ which converge to $p$ and $q$ 
%respectively with $p$ in leaf $l$ of $\fnsL$, $q$ in
%leaf $r$ of $\fnsL$ and $r$ different from $l$. The last fact
%implies that the length of $l_n$ between $p_n$ and $q_n$ must 
%be converging to infinity. But the distance between them in $L$ is
%bounded, since $p_n, q_n$ converge to $p, q$. That would imply
%the leaves $l_n$ are not uniform quasigeodesics, contradiction.
Recall that the leaf space of 
%$\fnsL$ is Hausdorff and hence that of 
$\oos$ is Hausdorff. 
Therefore the leaf space $\hhs$ of $\oos$ is
a topological tree \cite{Fe3}.
The same happens for the leaf space of $\oou$.

\vskip .07in
Given any $g$ in $\pi_1(M)$ it induces a homeomorphism of this
topological tree $\hhs$.
$\bf Z$ actions on such trees are  well understood \cite{Ba,Fe3,Ro-St}.
There are 2 options:

$-$ $g$ acts freely and has an axis $v$. Elements in the
axis are those $z$ in $\hhs$ for which $g(z)$ separates $z$ from
$g^2(z)$, or

$-$ $g$ fixes a point in $\hhs$.
\vskip .07in

Suppose first that $g$ acts freely on $\hhs$. Then $g$ has an axis  $v$
for its action on $\hhs$ and consequently an axis for its
action on the leaf space of $\fnsL$.
Let $l$ be a leaf of $\fnsL$ in the axis and we may assume that
$l$ is non singular.
By the axis properties it follows that the leaves

$$\{ g^n(l),  \ \ n \in {\bf Z} \}$$

\noindent
are nested in $L$ and they are uniform quasigeodesics.
Since they escape when viewed in the leaf space of $\fnsL$, the
same is true in $L$.
As they are uniform quasigeodesics and nested, then there are 
unique points $p, q$ in $\pin L$ so that $g^n(l)$ converges to $p$
if $n$ converges to infinity and to $q$ if $n$ converges
to minus infinity. Hence under the identification of $\uu$ with
$\pin L$, then $p, q$ are the unique fixed points of (any power of)
$g$ in $\uu$, where $p$ is attracting and $q$ repelling. 
%The same
%happens for the action in $\uu$.
In this case the action of $g$ in $\oo$ could be orientation
preserving or not.
This is case II, 1).

\vskip .1in
From now on in the proof we assume that $g$ has a fixed point
in $\hhs$, so there is a leaf $l$ of $\hhs$ with $g(l) = l$.
Then the leaf 

$$\Theta(l) \ \ {\rm of} \ \ \oos \ \ \ {\rm contains \ a \ unique} \ \ x 
\ \ {\rm in} \ \ \oo \ \ {\rm with} \ \ g(x) \ = x$$

\noindent
If $g$ has no fixed points in $\uu$ then it acts as an orientation
preserving homeomorphism on $\uu$ and hence the same happens for the action
on $\oo$.

There is a smallest positive integer $i_0$ so that $h = g^{i_0}$
leaves invariant all prongs of $\oos(x), \oou(x)$.
If there are $2n$ such prongs, each generates an ideal point of $L$
and also a point of $\uu$. By lemma \ref{brick}
any two distinct prongs have different ideal points in $\uu$.
Hence $h$ has at least $2n$ fixed points in $\uu$.
Let $\alpha$ be the flow line of $\wwp$ with $\Theta(\alpha) = x$.
Without loss of generality assume that the prongs above are
circularly ordered with corresponding ideal points

$$a_1, b_1, ..., a_n, b_n \ \  {\rm in} \ \ \uu \ \ {\rm where} \ \ \ 
\partial \oos(x) = \{ a_1, a_2, ..., a_n \}, \ \ 
\partial \oou(x) = \{ b_1, b_2, ..., b_n \}.$$

\noindent
Suppose that $g$ is associated to the positive flow direction in $\alpha$.
Fix  a prong $\tau$ of $\oos(x)$ and let $I$ be the maximal interval of
$\uu - \partial \oou(x)$ containing the ideal point of $\tau$.
Let now $\mu$ be an arbitrary unstable leaf of $\oou$ intersecting
$\tau$.
Then as $\mu$ gets closer to prongs of $\oou(x)$, the ideal points of
$\mu$ approach the endpoints of $I$. The action of $h$ on
$\tau$ is as follows: $h$ fixes $x$ and for a leaf $\mu$ as above then
$h$ takes  it to a leaf farther away from  $x$.
This is because in $\mi$ the flow lines along stable leaves move
closer in forward time. 
Notice that $h$ acts on $\mi$ as an isometry. This isometry
takes a flow line in $\wws(\alpha)$ to
one which is at the same distance from $\alpha$, but farther away from
$\alpha$ than $\beta$ is $-$ as $\beta$ is getting closer to $\alpha$.
Flowing back to the original position it means that
the image is farther from the orbit $\alpha$ thatn $\beta$ is.
It follows that $h$ acts as an expansion in $\tau$ with a single fixed point
in $x$.
Given $\mu$ as above then $h^n(\mu) \cap \tau$ escapes in $\tau$ as
$n$ converges to infinity.
These also form a nested collection of leaves.
If the sequence $h^n(\mu)$ does not escape compact sets in $\oo$, then
it limits in a collection

$${\cal W} = \{ W_i, i \in C \}$$

\noindent
of leaves of $\oou$, where $C$ is an interval in ${\bf Z}$ either
finite or all of ${\bf Z}$ \cite{Fe6}.
In addition $h$ leaves invariant ${\cal W}$. 
If ${\cal W}$ is not finite, then in particular it is not a single point
and then the leaf space of $\oos$ is not Hausdorff, which is
impossible as seen previously.
If on the other hand ${\cal W}$  is a single leaf $W$, then $h(W) = W$ and
there is a single periodic point $z$ in $W$ with $h(z) = z$. Then $h$ fixes
$x$ and $z$ and this is also impossible as seen above.

It follows that $h^n(\mu)$ escapes compact sets in $\oo$ and 
as seen in the free action case, they can only limit in a single point of $\uu$,
which corresponds to the ideal point $p$ of $\tau$.
This shows that $h$ acts as a contraction in $I$ with fixed point $p$.
Hence the points $a_i, 1 \leq i \leq n$ are attracting fixed points
of $h$ in $\uu$.
Using $h^{-1}$ one shows that the $b_i, 1 \leq i \leq n$ are repelling fixed points
and these are the only fixed points of $h$ in $\uu$.
%Since this works for all prongs of $\oou(x)$ it follows that
Hence $h$ fixes exactly $2n$ points in $\uu$, where $n \geq 2$.

\vskip .1in
We now return to $g$. If $g$ is orientation reversing on $\uu$, then so is the
action on $\oo$. In this case there are exactly $2$ fixed points of $g$ in
$\uu$.
The square of $g$ is now orientation preserving on $\uu$ and it has
fixed points. In particular any fixed point of $g^{2i}$ is a fixed point
of $g^2$. It follows that $h$ is equal to $g^2$ and this is case II, 2).

Suppose finally that $g$ is orientation preserving on $\uu$.
Since $h = g^{i_0}$ has fixed points in $\uu$, then either $g$ has
no fixed points in in $\uu$ or $g$ has exactly the same fixed points
in $\uu$ as $h$ does. 
In the second case $h$ is equal to $g$ and
$g$ has exactly $2n$ fixed points
in $\uu$, which are alternatively attracting and contracting.
This is case I).
In the first case $g$ acts essentially as a rotation in
$\uu$ and $\oo$. 
This is case III).

This finishes the proof of the proposition.
\end{proof}

\section{Construction of the conjugacy}

Let now $\Phi, \Psi$ be two pseudo-Anosov flows transverse to the $\rrrr$-covered
foliation $\gal$ and both regulating for $\gal$.
We want to show that $\Phi$ and $\Psi$ are topologically conjugate.
Let $\oo$ be the orbit space of $\wwp$ and $\tz$ be the orbit space of 
$\widetilde \Psi$.
We will construct a $\pi_1(M)$-equivariant homeomorphism from $\oo$ to $\tz$. We first associate
to each periodic orbit of $\Phi$ a unique periodic orbit of $\Psi$.
Let

$$\Theta_1: \ \mi \ \rightarrow \ \oo \ \ \ \ \ {\rm and} \ \ \  \ \ 
\Theta_2: \ \mi \ \rightarrow \ \tz$$

\noindent
be the corresponding orbit space projection maps.
Let $\oos, \oou$ be the projection stable and unstable foliations of $\wwp$ to $\oo$ and
$\tzs, \tzu$ the corresponding objects for $\widetilde \Psi$.
Recall that $\pi: \mi \rightarrow M$ is the universal covering map.

One main property to note here is that the universal circle $\uu$ depends
only on $\gal$ and not on $\Phi$ or $\Psi$. The same is true for the action
of $\pi_1(M)$ on $\uu$.
Before we prove the theorem, we first prove a preliminary property:

\begin{lemma}{}{}
Let $\alpha$ be an orbit of $\wwp$ so that $\pi(\alpha)$ is a closed orbit
of $\Phi$. Let $g$ be the element of $\pi_1(M)$ associated
to the closed orbit $\pi(\alpha)$. 
 Then there is a unique closed orbit $\beta$ of $\widetilde \Psi$
so that $\pi(\beta)$ is periodic and associated to $g$, that is,
$\pi(\beta)$ is freely homotopic to $\pi(\alpha)$. 
\end{lemma}

\begin{proof}{}
Let $x = \Theta_1(\alpha)$ and $g$ non trivial in $\pi_1(M)$ with $g(x) = x$ and
indivisible with respect to this property. Suppose that $g$ is associated
to the forward flow direction of $\pi(\alpha)$.
Let $h$ be the smallest power of $g$ so that $h$ leaves invariant
all prongs of $\oos(x), \oou(x)$.
Proposition \ref{act} shows that $h$ has $2n$ fixed points in $\uu$, with $n \geq 2$.
This is case I).
Now apply this proposition to $h$ and $\Psi$. 
Since $h$ has $2n$ fixed points in $\uu$ and $n \geq 2$, proposition \ref{act} implies 
that there is a unique $y$ in $\tz$
with $h(y) = y$.
Let 

$$\beta \ \ {\rm be \ the \ orbit \ of} \ \ \widetilde \Psi
\ \ {\rm with} \ \ \Theta_2(\beta) \ = \ y, \ \ {\rm so} \ \ 
h(\beta) \ = \beta$$

\noindent
If $g$ acts freely on  $\tz$ then the analysis of  proposition \ref{act} shows that
$h$ can have only 2 fixed points in $\uu$, impossible.
It follows that 
$g$ cannot act freely on $\tz$ and
therefore the only fixed point of $g$ in $\tz$ is $y$ $-$ as it is fixed
under a power of $g$. This implies that $g(\beta) = \beta$ and consequently
$\pi(\alpha)$ is freely homotopic to a power of $\pi(\beta)$.
Reversing the roles of $\alpha$ and $\beta$ implies that
$\pi(\alpha)$ and $\pi(\beta)$ are freely homotopic to each other
or their inverses.
The action of $h$ on $\uu$ shows that the first option is the one that happens 
$-$ this is because they both have attracting fixed points in $\uu$ in the
same points.
This finishes the proof of the lemma.
\end{proof}

This defines a map from the periodic points of $\oo$ to the periodic points
of $\tz$.
Notice that in the lemma above $\partial \oos(x) = \partial \tzs(y)$ as points in $\uu$
and similarly for $\oou(x), \tzu(y)$.

\begin{theorem}{}{}
Let $\Phi, \Psi$ be pseudo-Anosov flows, which are transverse and regulating
for an $\rrrr$-covered foliation $\gal$. Then $\Phi, \Psi$ are 
topologically conjugate.
\label{main}
\end{theorem}

\begin{proof}{}
We first define a map $f$ from $\oo$ to $\tz$ which extends the correspondence
between periodic points  obtained previously.
Given $x$ in $\oo$, we will let $y$ be the unique point of $\tz$ with

$$\partial \tzs(y) \ = \ \partial \oos(x), \ \ \ 
\partial \tzu(y) \ = \ \partial \oou(x)$$

If $x$ is periodic, this was constructed in the previous lemma.
The previous lemma shows that there is always such a $y$. Lemma
\ref{brick} shows that $y$ is uniquely defined.
This is because the set $\partial \tzs(y)$ determines the stable
leaf  $\tzs(y)$ and likewise for $\tzu(y)$. Hence $y$ is uniquely defined.
If $x$ is not periodic let $x_n$ in $\oo$ which are periodic
and converging to $x$.
We may assume that no $x_n$ is singular since the singular orbits form
a discrete subset of $\oo$.
We can also assume that $(\oos(x_n))$ forms a nested sequence, and so
does $(\oou(x_n))$.
Let $p_n, q_n$ points in $\uu$ with 

$$\partial \oos(x_n) \ = \ \{ p_n, q_n \} \ \ \ {\rm and \ let} \ \ \ \ \{p, q \} \ = \ \partial \oos(x)$$

\noindent
Then up to renaming we can assume that $p_n$ converges to $p$
in $\uu$ and $q_n$ converges to
$q$ in $\uu$. Let $y_n$ in $\tz$ periodic
with $\partial \tzs(y_n) = \{ p_n, q_n \}$. Notice that the $l_n = \tzs(y_n)$ are leaves
of $\tzs$, which are nested in $\tz$. By the identification of $L$ with $\oo$,
then the $l_n$ are associated to uniform
quasigeodesics in $L$ which have ideal points which converge to distinct
points in $\pin L$ (associated to $p, q$ in $\uu$).
Therefore these quasigeodesics converge to a single quasigeodesic in $L$
and so 

$$\tzs(y_n) \ \ \ {\rm converges \ to \ a \ leaf} \ \ l \ \ {\rm of} \ \ \tzs$$

\noindent
Similarly
$\tzu(y_n)$ converges to a leaf $s$ of $\tzu$.
For all $n$, the pairs $\partial \oos(x_n)$, $\partial \oou(x_n)$ link each other in $\uu$,
so the same happens for $\partial \tzs(y_n), \partial \tzu(y_n)$.
It follows that the ideal points of $l$, $s$ link each other in $\uu$ 
otherwise we  would have a leaf of $\tzs$ sharing an ideal point
with a leaf of $\tzu$ $-$ which is disallowed by lemma \ref{brick}.
Therefore

$$y_n \ =  \ \tzs(y_n) \cap \tzu(y_n)$$

\noindent
converges to a point $y$ in $\tz$.
Clearly $\partial \tzs(y)$ contains $\partial \oos(x)$ and similarly
$\partial \tzu(y)$ contains $\partial \oou(x)$.
If $x$ is a singular orbit, one could apply the inverse process
to produce $x'$ in $\oo$, $x'$ singular so that 
$\partial \oos(x') = \partial \tzs(y)$. But then $\partial \oos(x')$ contains
$\partial \oos(x)$ and $x$ is non singular. This is disallowed by the
lemma. Therefore $y$ is non singular and hence

$$\partial \tzs(y) \ = \ \partial \oos(x), \ \ \ 
\partial \tzu(y) \ = \ \partial \oou(x)$$

In addition, since no two stable leaves of $\oos$ can share an ideal point,
then $y$ is well defined.

This map $f: \oo \rightarrow \tz$ is well defined. It is also 
injective. If $f(x_1) = f(x_2)$ then 
$\partial \oos(x_1) = \partial \oos(x_2)$ and $\partial \oou(x_1) = \partial \oou(x_2)$.
By lemma \ref{brick}, the first fact implies that $\oos(x_1) = \oos(x_2)$ and the second
fact implies that $\oou(x_1) = \oou(x_2)$.
Therefore their intersection is $x_1 = x_2$ and the map $f$ is injective.
In addition, the map $f$ clearly has an inverse by doing the same procedure 
from $\Psi$ to $\Phi$. Therefore $f$ is a bijection.

We claim that $f$ is continuous and by symmetry, then the inverse will also
be continuous.
Let then $x$ in $\oo$ and $(x_n)$ a sequence in $\oo$ converging to $x$.
Assume first that $x$ is non singular. Then 

$$\oos(x_n) \ \ {\rm conveges \
to} \ \ \oos(x) \ \ {\rm and} \ \ \partial \oos(x_n) \ \ {\rm converges \ to} \ \ 
\partial \oos(x) \ \ {\rm in} \ \ \uu$$

\noindent
Hence $\partial \tzs(f(x_n))$ converges to $\partial \tzs(f(x))$ and
similarly for $\partial \tzu(f(x_n))$.
This shows that $f(x_n)$ converges to $f(x)$ in $\tz$.

Suppose finally that $x$ is singular. Up to subsequence we may assume that 
$(x_n)$ are all in a sector of $\oos(x)$ bounded by the line leaf $l$
(contained in $\oos(x)$). Then $\oos(x_n)$ converges to $l$ and
$\partial \oos(x_n)$ converges to $\partial l$ in $\uu$.
It follows that 

$$\partial \tzs(f(x_n)) \ \ {\rm converges \ to} \ \  \partial l \ - \  {\rm a \ subset \ of} \ \  \uu$$

\noindent
which is contained in $\partial \tzs(f(x))$.
The same happens if $x_n$ are in $\oos(x)$, that is, $\partial \tzs(f(x_n))$
are contained in $\partial \tzs(f(x))$.
This shows that $\partial \tzs(f(x_n))$ only accumulates in $\partial \tzs(f(x))$.
The same is true for $\partial \tzu(f(x_n))$, which only accumulates in $\partial \tzu(f(x))$.
Then

$$f(x_n) \ \  = \ \  \tzs(f(x_n)) \ \cap \ \tzu(f(x_n))$$

\noindent
only accumulates in 

$$f(x) \ \  = \ \  \tzs(f(x)) \ \cap \ \tzu(f(x))$$

This shows that $f(x_n)$ has to converge to $f(x)$. 
This shows that $f$ is a homeorphism from $\oo$ to $\tz$.

In addition $f$ is $\pi_1(M)$ equivariant: If $g$ is in $\pi_1(M)$
and $x$ is in $\oo$, then 

$$g(\oos(x)) \ = \ \oos(g(x)), \ \ 
g(\oou(x)) \ = \ \oou(g(x))$$

\noindent
have ideal  points

$$\partial \oos(g(x)) \ \ \ {\rm and} \ \ \ \partial \oou(g(x))$$

\noindent
respectively.
Hence these are also the ideal points of 

$$\tzs(f(g(x))),  \ \ \tzu(f(g(x)))$$

\noindent
In addition

$$ \partial \tzs(f(x)) \ = \ \partial \oos(x) \ \ \ \ \ {\rm and} \ \ \ \ \ 
\partial g(\oos(f(x))) \ = \ 
\partial \tzs(g(f(x)))$$

\noindent
Hence they are the same as $\tzs(f(g(x))$.
Since this is also true for the unstable foliations, it follows
that 

$$f(g(x)) \ = \ g(f(x)) \ \ \ - \ \ \ \pi_1(M) \ {\rm equivariance}.$$

\vskip .1in
We now finish the proof of topological conjugacy between $\Phi$ and $\Psi$. We define a map
$h: \mi \rightarrow \mi$ as follows. Given $p$ in $\mi$, then $p$ is in a leaf
$L$ of $\gn$. Define

$$h(p) \ \ = \ \ \widetilde \Psi_{\rrrr}(f(\Theta(p))) \cap L,$$

\noindent
here $\Theta(p)$ is in $\oo$ and $f(\Theta(p))$ is in ${\cal T}$.
Essentially we look at the orbit $\alpha = \wwp_{\rrrr}(p)$ of the flow $\wwp$ through $p$
and consider the corresponding orbit of $\widetilde \Psi$ under the map $f$: that is the 
orbit 
$\widetilde \Psi_{\rrrr}(f(\Theta(p)))$ 
of $\widetilde \Psi$. Then we intersect this orbit of $\widetilde \Psi$ with $L$. 
This map $h$ preserves
the leaves of $\gn$ $-$ not just the foliation $\gn$, but the leaves themselves.
In addition $h$ sends
orbits of $\wwp$ to orbits of $\widetilde \Psi$.
The map $h$ is clearly continuous and hence defines a homeomorphism of $\mi$.
From the equivariance of $f$ it follows that $h$ is also equivariant,
that is for any $g$  and $p$ in $\mi$, then 
$h(g(p))  = g(h(p))$.
Therefore $h$ induces a homeomorphism of $M$, which sends orbits of $\Phi$ to
orbits of $\Psi$. hence $\Phi$ and $\Psi$ are topologically
conjugate. This finishes the proof of theorem \ref{main}.
\end{proof}

{\footnotesize
{
\setlength{\baselineskip}{0.01cm}

\noindent
Florida State University

\noindent
Tallahassee, FL 32306-4510

}
}

\end{document}